\documentclass[11pt]{amsart}
\usepackage[T1]{fontenc}
\usepackage[utf8]{inputenc}
\usepackage{amsmath,amssymb,amsfonts,amsthm,mathrsfs,mathtools}
\usepackage{cite}
\usepackage[dvipsnames]{xcolor}
\usepackage[colorlinks=true,linkcolor=blue,citecolor=red,urlcolor=blue]{hyperref}

\usepackage{amsfonts,enumerate,verbatim,mathtools,tikz,bm,tikz-cd,hyperref,comment}
\numberwithin{equation}{section}
\newtheorem{theorem}{Theorem}[section]
\newtheorem{proposition}[theorem]{Proposition}
\newtheorem{lemma}[theorem]{Lemma}
\newtheorem{corollary}[theorem]{Corollary}
\theoremstyle{definition}
\newtheorem{definition}[theorem]{Definition}
\newtheorem{example}[theorem]{Example}
\theoremstyle{remark}
\newtheorem{remark}[theorem]{Remark}
\newcommand{\cA}{\mathcal A}
\newcommand{\cB}{\mathcal B}
\newcommand{\cC}{\mathcal C}
\newcommand{\cD}{\mathcal D}
\newcommand{\cE}{\mathcal E}

\newcommand{\cM}{\mathcal M}
\newcommand{\cN}{\mathcal N}
\newcommand{\cQ}{\mathcal Q}
\newcommand{\cU}{\mathcal U}
\newcommand{\cV}{\mathcal V}

\newcommand{\cX}{\mathcal X}
\newcommand{\cY}{\mathcal Y}
\newcommand{\sF}{\mathsf F}
\newcommand{\sI}{\mathsf I}
\newcommand{\sT}{\mathsf T}
\newcommand{\sU}{\mathsf U}
\newcommand{\sZ}{\mathsf Z}
\newcommand{\sC}{\mathsf C}

\newcommand{\Ext}{\operatorname{Ext}}
\newcommand{\Tor}{\operatorname{Tor}}

\newcommand{\Coker}{\operatorname{Coker}}
\newcommand{\Img}{\operatorname{Im}}

\newcommand{\Mod}{\operatorname{Mod}}

\newcommand{\ellength}{\ell_{\eta}}

\title[Higher Cotorsion Pairs on Image-Nilpotent
\(\eta\)-Extensions]
{Higher Cotorsion Pairs on Image-Nilpotent
\(\eta\)-Extensions of Abelian Categories}

\author[D. Liu]{Dajun Liu}
\address{School of Mathematics-Physics and Finance,
Anhui Polytechnic University,
Wuhu 241000, Anhui, P. R. China}
\email{liudajun@ahpu.edu.cn}

\author[H. Gao]{Hanpeng Gao}
\address{School of Mathematical Sciences,
Anhui University,
Hefei 230601, Anhui, P. R. China}
\email{hpgao07@163.com}

\subjclass[2020]{16B50; 16E30; 18G25}


\keywords{\(\eta\)-extension, \(n\)-cotorsion pair, \(\eta\)-Loewy filtration,
image-nilpotent extension }

\begin{document}
\maketitle

\begin{abstract}
Let \(\cB\) be an abelian category with enough projective and injective
objects, and let
$
\cA=\cB\ltimes_\eta\sF
$
be an \(\eta\)-extension induced by a right exact endofunctor \(\sF\).
We introduce uniform image-nilpotence for \(\eta\)-extensions and
characterize it in terms of the nilpotence of the associative natural
transformation \(\eta:\sF^2\to\sF\). Using the associated \(\eta\)-Loewy filtration, we prove a
lifting theorem for right \(n\)-cotorsion pairs over
image-nilpotent \(\eta\)-extensions and establish a corresponding
completeness theorem. As applications, our results
recover and unify related constructions for split nilpotent ring
extensions, comma categories, formal triangular matrix rings, Morita
context rings and classical trivial extensions.
\end{abstract}
\maketitle
\section{Introduction}
Cotorsion pairs organize approximation theory, relative homological algebra and abelian model structures. Their behavior under matrix constructions and categorical extensions has therefore been studied in several settings, including formal triangular matrix rings, comma categories, Morita context rings and trivial extension rings, see \cite{MaoTriangular,HuZhu,YuanHeWu,CuiRongZhang,MaoTrivial}. Marmaridis introduced $\eta$-extensions of abelian categories in \cite{Marmaridis}, and Beligiannis developed the broader cleft-extension framework and its relative homology in \cite{BeligiannisCleft,BeligiannisRelative}. Recently, Hu proved that a cotorsion pair $(\cX,\cY)$ in an abelian category $\cB$ induces a cotorsion pair
\[
({}^{\perp}\sU^{-1}(\cY),\sU^{-1}(\cY))
\]
in an $\eta$-extension $\cB\ltimes_{\eta}\sF$, and studied heredity, completeness, Gorenstein projective objects and Hovey triples \cite{Hu}. 

The higher analogue of a cotorsion pair was introduced by Huerta, Mendoza and P\'erez \cite{HMP}. Long and Zhang constructed left and right $n$-cotorsion pairs over formal triangular matrix rings \cite{LongZhang}. Their work demonstrates that higher cotorsion theory is compatible with matrix extensions. The purpose of the present paper is to isolate the structural mechanism which permits a genuinely non-trivial $\eta$-extension theorem.

The starting point of the present paper is the image-cokernel
decomposition \cite[Lemma 3.14]{Hu}. By iterating its image term, we
obtain a descending filtration
\[
A=\sI^0(A)\supseteq\sI(A)\supseteq\sI^2(A)\supseteq\cdots,
\]
called the \(\eta\)-Loewy filtration. We prove that this filtration is
governed by the iterated multiplication of \(\eta\). More precisely,
for every \(d\geq1\),
\[
\cA\text{ is uniformly image-nilpotent of index at most }d
\quad\Longleftrightarrow\quad
\eta^{[d]}=0.
\]
Thus image-nilpotence provides the finite structural mechanism needed
to pass from the objects in the image of the zero embedding
\(\sZ:\cB\to\cA\) to arbitrary objects of \(\cA\).

Our main result is a lifting theorem for right
\(n\)-cotorsion pairs. It shows that, on an image-nilpotent
\(\eta\)-extension, the approximation property for all objects is
determined by the corresponding property on the objects \(\sZ(B)\).
As a functorial consequence, suitable right \(n\)-cotorsion pairs
\((\cX,\cY)\) in \(\cB\) induce right \(n\)-cotorsion pairs
$
\bigl(\sT(\cX),\sU^{-1}(\cY)\bigr)
$
in \(\cA\). The required approximations are obtained by lifting the
basic approximations in \(\cB\) and gluing them along the finite
\(\eta\)-Loewy filtration.

We also prove that if \((\cX,\cY)\) is a hereditary complete cotorsion
pair satisfying
\[
L_1\sF(\cX)=0,\qquad \sF(\cX)\subseteq\cY,
\]
then
$
\bigl({}^{\perp}\sU^{-1}(\cY),\sU^{-1}(\cY)\bigr)
$
is a hereditary complete cotorsion pair over every image-nilpotent
\(\eta\)-extension. This extends the completeness result of \cite{Hu}.

The general results specialize to split nilpotent ring extensions,
comma categories, formal triangular matrix rings, Morita context
rings and classical trivial extensions. In this way, the
image-nilpotent framework gives a common explanation for several
cotorsion-pair constructions appearing in
\cite{Hu,LiuFengHuZhang,LongZhang,MaoTriangular,CuiRongZhang,
MaoTrivial}.

The paper is organized as follows. Section 2 recalls the necessary
facts on \(\eta\)-extensions, exact equivalences and \(n\)-cotorsion
pairs. Section 3 develops the \(\eta\)-Loewy filtration, establishes
the nilpotence criterion and proves the filtration reduction theorem.
Section 4 will construct  right \(n\)-cotorsion pairs
$
\bigl(\sT(\cX),\sU^{-1}(\cY)\bigr)
$
in \(\cA\) and then 
contain the completeness theorem.
Section 5 applies the general theory to split nilpotent ring
extensions, comma categories, formal triangular matrix rings, right
trivial extensions and Morita context rings.

\section{Preliminaries}
Throughout, all subcategories are full, additive and closed under isomorphisms. The abelian categories under consideration have enough projective and injective objects.

Let $\cB$ be an abelian category, $\sF:\cB\to\cB$ be a covariant right exact additive functor and  $\eta:\sF^2\to\sF$ be an associative natural transformation. Assume that $\eta$ is \emph{associative}, that is, $\eta\circ\mathsf{F}\eta=\eta\circ\eta\mathsf{F}$.
The $\eta$-extension $\cA=\cB\ltimes_{\eta}\sF$ has objects $(B,f)$ with $f:\sF(B)\to B$ satisfying
\[
f\sF(f)=f\eta_B,
\]
and morphisms $\alpha:(B,f)\to(B',f')$ satisfying $\alpha f=f'\sF(\alpha)$. This is an abelian category, see \cite{Marmaridis}. If $\eta=0$, the category is denoted by $\cB\ltimes\sF$ and is called a right trivial extension.

There are the following functors:
\begin{itemize}
	\item $\mathsf{U}(X,f)=X$ and $\mathsf{U}(\alpha)=\alpha$;
	\item $\mathsf{Z}(X)=(X,0)$ and $\mathsf{Z}(\alpha)=\alpha$;
	\item $\mathsf{T}(X)=\left(X\oplus\mathsf{F}(X),t_X:=\left(\begin{smallmatrix}
		0&0\\
		1&\eta_X
	\end{smallmatrix}\right)\right)$ and $\mathsf{T}(\alpha)=\begin{pmatrix}
		\alpha&0\\
		0&\mathsf{F}(\alpha)
	\end{pmatrix}$;
	\item $\mathsf{C}(X,f)=\mathrm{Coker}(f)$ and $\mathsf{C}(\alpha)$ is the morphism induced by the cokernel.
\end{itemize}

The associativity of $\eta$ guarantees that $\mathsf{T}(X)$ is indeed an object of $\mathcal{B}\ltimes_\eta\mathsf{F}$ (the converse is also true). The following Lemma is  obvious.

\begin{lemma}\label{lem:1}
	Let $\mathcal{B}\ltimes_\eta\mathsf{F}$ be an $\eta$-extension of $\mathcal{B}$. Then the  sequence $(X,f)\xrightarrow{\alpha}(Y,g)\xrightarrow{\beta}(Z,h)$ is exact in $\mathcal{B}\ltimes_\eta\mathsf{F}$ if and only if the underlying sequence$ X\xrightarrow{\alpha}Y\xrightarrow{\beta}Z$ is exact in $\mathcal{B}$. In particular, the functor $\mathsf{U}$ is exact.
\end{lemma}

\begin{definition}\label{def:higher-orthogonals}
Let \(\mathcal E\) be an abelian category, let \(\mathcal M\) be a class of objects of \(\mathcal E\), and let \(1\leq a\leq b\). Define
\[
\mathcal M^{\perp_{[a,b]}}
=
\left\{
N\in\mathcal E\ \middle|\
\Ext^i_{\mathcal E}(M,N)=0
\text{ for all }M\in\mathcal M,\ a\leq i\leq b
\right\},
\]
\[
{}^{\perp_{[a,b]}}\mathcal M
=
\left\{
N\in\mathcal E\ \middle|\
\Ext^i_{\mathcal E}(N,M)=0
\text{ for all }M\in\mathcal M,\ a\leq i\leq b
\right\}.
\]
For \(a=b=1\), we use the abbreviations
\[
\mathcal M^\perp=\mathcal M^{\perp_{[1,1]}},
\qquad
{}^\perp\mathcal M={}^{\perp_{[1,1]}}\mathcal M.
\]
For classes \(\mathcal X,\mathcal Y\subseteq\mathcal E\), the notation
\[
\Ext^i_{\mathcal E}(\mathcal X,\mathcal Y)=0
\]
means that \(\Ext^i_{\mathcal E}(X,Y)=0\) for all
\(X\in\mathcal X\) and \(Y\in\mathcal Y\).
\end{definition}

\begin{definition}[\cite{Hu}]\label{def:cotorsion-pair}
Let \(\mathcal E\) be an abelian category. A pair
\((\mathcal X,\mathcal Y)\) of classes of objects of \(\mathcal E\)
is called a cotorsion pair if
\[
\mathcal X={}^{\perp_{}}\mathcal Y,
\qquad
\mathcal Y=\mathcal X^{\perp_{}}.
\]

An epimorphism \(X_A\to A\) is called a special
\(\mathcal X\)-precover of \(A\) if it occurs in a short exact sequence
\[
0\to Y_A\to X_A\to A\to0
\]
with \(X_A\in\mathcal X\) and \(Y_A\in\mathcal Y\). A monomorphism
\(A\to Y^A\) is called a special \(\mathcal Y\)-preenvelope of \(A\)
if it occurs in a short exact sequence
\[
0\to A\to Y^A\to X^A\to0
\]
with \(Y^A\in\mathcal Y\) and \(X^A\in\mathcal X\).

The cotorsion pair is complete if every object of \(\mathcal E\)
admits both a special \(\mathcal X\)-precover and a special
\(\mathcal Y\)-preenvelope. It is hereditary if
\[
\Ext^i_{\mathcal E}(\mathcal X,\mathcal Y)=0
\qquad\text{for all }i\geq1.
\]
A cotorsion pair which is both complete and hereditary is called a
hereditary complete cotorsion pair.
\end{definition}

\begin{definition}\label{def:resolving-coresolving}
Let \(\mathcal E\) be an abelian category and let
\(\mathcal M\) be a full additive subcategory of \(\mathcal E\).
The subcategory \(\mathcal M\) is called resolving if it contains
all projective objects, is closed under extensions, and is closed
under kernels of epimorphisms.
Dually, \(\mathcal M\) is called coresolving if it contains all
injective objects, is closed under extensions, and is closed under
cokernels of monomorphisms.
\end{definition}

\begin{proposition}[\cite{CLZ,EnochsJenda,G11,GobelTrlifaj}]\label{prop:cotorsion-criteria}
Let \((\mathcal X,\mathcal Y)\) be a cotorsion pair in an abelian
category \(\mathcal E\) with enough projective and injective objects.
\begin{enumerate}
\item The following conditions are equivalent:
\begin{enumerate}
\item \((\mathcal X,\mathcal Y)\) is complete;
\item every object of \(\mathcal E\) admits a special
\(\mathcal X\)-precover;
\item every object of \(\mathcal E\) admits a special
\(\mathcal Y\)-preenvelope.
\end{enumerate}
\item The following conditions are equivalent:
\begin{enumerate}
\item \((\mathcal X,\mathcal Y)\) is hereditary;
\item \(\Ext^2_{\mathcal E}(\mathcal X,\mathcal Y)=0\);
\item \(\Ext^i_{\mathcal E}(\mathcal X,\mathcal Y)=0\) for all
\(i\geq1\);
\item \(\mathcal X\) is closed under kernels of epimorphisms;
\item \(\mathcal Y\) is closed under cokernels of monomorphisms.
\end{enumerate}
\end{enumerate}
\end{proposition}

\begin{lemma}[\cite{EnochsJenda,GobelTrlifaj}]\label{lem:cotorsion-closure}
If \((\mathcal X,\mathcal Y)\) is a cotorsion pair in an abelian
category, then both \(\mathcal X\) and \(\mathcal Y\) are closed under
extensions and direct summands.
\end{lemma}

We write $\cM^{\vee}_m$ for the objects admitting a coresolution
\[
0\longrightarrow N\longrightarrow M^0\longrightarrow\cdots\longrightarrow M^m\longrightarrow0
\]
with $M^j\in\cM$, and $\cM^{\wedge}_m$ for the dual resolution class.

\begin{definition}\label{def:n-cotorsion-pairs}
Let \(n\geq1\), and \((\cX,\cY)\) be classes in an
abelian category \(\cB\).

The pair \((\cX,\cY)\) is a left \(n\)-cotorsion pair if:
\begin{enumerate}
\item \(\cX\) is closed under direct summands;
\item \(\Ext^i_{\cB}(\cX,\cY)=0\) for \(1\leq i\leq n\);
\item every \(B\in\cB\) admits a short exact sequence
\[
0\to Y_B\to X_B\to B\to0
\]
with \(X_B\in\cX\) and \(Y_B\in\cY^\wedge_{n-1}\).
\end{enumerate}

The pair is a right \(n\)-cotorsion pair if:
\begin{enumerate}
\item \(\cY\) is closed under direct summands;
\item \(\Ext^i_{\cB}(\cX,\cY)=0\) for \(1\leq i\leq n\);
\item every \(B\in\cB\) admits a short exact sequence
\[
0\to B\to Y_B\to X_B\to0
\]
with \(Y_B\in\cY\) and \(X_B\in\cX^\vee_{n-1}\).
\end{enumerate}
\end{definition}

A left or right \(n\)-cotorsion pair \((\cX,\cY)\) is called
hereditary if
$
\Ext^i_{\cB}(\cX,\cY)=0
$
for every \(i\geq1\).

The following characterization and dimension-shifting property will be used repeatedly.

\begin{proposition}\label{prop:HMP}
Let $(\cX,\cY)$ be a pair of classes in $\cB$.
\begin{enumerate}
\item The pair is a right $n$-cotorsion pair if and only if
\[
\cY=\cX^{\perp_{[1,n]}}
\]
and every $B\in\cB$ admits the exact sequence in Definition \ref{def:n-cotorsion-pairs}.
\item If $\Ext^i_{\cB}(\cX,\cY)=0$ for $1\leq i\leq n$, then
\[
\Ext^1_{\cB}(\cX^{\vee}_{n-1},\cY)=0.
\]
\end{enumerate}
\end{proposition}

\begin{proof}
These are \cite[Theorem 2.7 and Proposition 2.5]{HMP}.
\end{proof}

In particular, the right class of a right $n$-cotorsion pair is closed under extensions, finite direct sums and direct summands. For $n=1$, the simultaneous left and right conditions recover complete cotorsion pairs; see \cite{HMP}. We shall also use the standard criterion that a cotorsion pair in an abelian category with enough projectives and injectives is complete once every object has a special preenvelope by its right class; see \cite{EnochsJenda,GobelTrlifaj} and \cite[Proposition 2.4]{Hu}.

\begin{definition}\label{def:exact-equivalence}
Let \(\mathcal E\) and \(\mathcal F\) be abelian categories. A functor
$
\Phi:\mathcal E\to\mathcal F
$
is called an exact equivalence if \(\Phi\) is fully faithful, essentially surjective, and sends every short exact sequence
$
0\to X'\to X\to X''\to0
$
in \(\mathcal E\) to a short exact sequence
$
0\to\Phi(X')\to\Phi(X)\to\Phi(X'')\to0
$
in \(\mathcal F\). In this case, a quasi-inverse of \(\Phi\) is a functor
$
\Psi:\mathcal F\to\mathcal E
$
together with natural isomorphisms
\[
\Psi\Phi\cong1_{\mathcal E},\qquad
\Phi\Psi\cong1_{\mathcal F}.
\]
See \cite[Chapter IV, Section 4]{MacLane}.
\end{definition}

\begin{lemma}\label{lem:exact-equivalence}
Let
\[
\Phi:\mathcal E\to\mathcal F
\]
be an exact equivalence of abelian categories, and let
\(\Psi:\mathcal F\to\mathcal E\) be a quasi-inverse. Then \(\Phi\)
and \(\Psi\) preserve and reflect short exact sequences, direct
summands, finite resolutions and finite coresolutions. Moreover, for
all \(X,Y\in\mathcal E\) and \(i\geq0\), there are natural
isomorphisms
\[
\Ext_{\mathcal F}^i(\Phi(X),\Phi(Y))
\cong\Ext_{\mathcal E}^i(X,Y).
\]
Consequently, exact equivalences preserve resolving and coresolving
subcategories, right and left \(n\)-cotorsion pairs, and complete or
hereditary cotorsion pairs, see \cite[Chapter IV, Section 4]{MacLane} and
\cite[Chapter 2]{Weibel}.
\end{lemma}

\section{The $\eta$-Loewy filtration and homological reduction}
We first extract the iterated structure in the image-cokernel decomposition \cite[Lemma 3.14]{Hu}.

\begin{definition}\label{def:iterated-action}
For $r\geq1$, define natural transformations $\eta^{[r]}:\sF^r\to\sF$ by
\[
\eta^{[1]}=1_{\sF},\qquad \eta^{[r+1]}=\eta\circ\sF(\eta^{[r]}).
\]
For $A=(B,f)\in\cA=\cB\ltimes_\eta\sF$, define morphisms $f^{[r]}:\sF^r(B)\to B$ by
\[
f^{[1]}=f,
\qquad
f^{[r+1]}=f\sF(f^{[r]}).
\]
\end{definition}

\begin{lemma}\label{lem:iterated-action}
Let $A=(B,f)\in\cA$. For all $r,p,q\geq1$ one has
\[
f^{[r]}=f\eta^{[r]}_B
\]
and
\[
f^{[p+q]}=f^{[p]}\sF^p(f^{[q]}).
\]
In particular, the morphism $f^{[r]}$ is independent of the placement of parentheses in the $r$-fold product of $f$.
\end{lemma}

\begin{proof}
The first equality is proved by induction. It is clear for $r=1$. If it holds for $r$, then
\[
f^{[r+1]}=f\sF(f\eta_B^{[r]})=f\sF(f)\sF(\eta_B^{[r]})=f\eta_B\sF(\eta_B^{[r]})=f\eta_B^{[r+1]}.
\]
The recursive definition also gives
\[
f^{[r]}=f\sF(f)\cdots\sF^{r-1}(f).
\]
Concatenating the factors yields $f^{[p+q]}=f^{[p]}\sF^p(f^{[q]})$.
\end{proof}

Let $A=(B,f)$. Write $f=i_Aq_A$ with $q_A:\sF(B)\twoheadrightarrow\Img f$ and $i_A:\Img f\rightarrowtail B$.

\begin{proposition}\label{prop:image-functor}
For \(A=(B,f)\in\cA=\cB\ltimes_\eta\sF\), write \(f=i_Aq_A\), where
\[
q_A:\sF(B)\twoheadrightarrow\Img f,\qquad i_A:\Img f\rightarrowtail B
\]
is the image factorization of \(f\). Let \(g_A=q_A\sF(i_A):\sF(\Img f)\to\Img f\). Then the assignment
\[
\sI(A)=\sI(B,f)=(\Img f,g_A)
\]
defines an additive functor \(\sI:\cA\to\cA\). Moreover, there is a natural short exact sequence of functors from \(\cA\) to \(\cA\):
\[
0\longrightarrow\sI\xrightarrow{\iota}1_{\cA}\xrightarrow{\rho}\sZ\sC\longrightarrow0,
\]
where \(1_{\cA}\) is the identity functor and \(\sZ\sC\) denotes the composite functor
\[
\cA\xrightarrow{\sC}\cB\xrightarrow{\sZ}\cA.
\]
For every \(A=(B,f)\), the component of this sequence is the following sequence
\[
0\longrightarrow(\Img f,g_A)\xrightarrow{i_A}(B,f)\xrightarrow{\rho_A}\sZ(\Coker f)\longrightarrow0.
\]
\end{proposition}

\begin{proof}
By \cite[Lemma 3.14]{Hu}, the pair \((\Img f,g_A)\) is an object of \(\cA\), and
\[
0\to(\Img f,g_A)\xrightarrow{i_A}(B,f)\xrightarrow{\rho_A}\sZ(\Coker f)\to0
\]
is exact in \(\cA\). We define \(\sI\) on morphisms. Let
\[
\alpha:(B,f)\to(B',f')
\]
be a morphism in \(\cA\). Thus \(\alpha f=f'\sF(\alpha)\). Write
\[
f=i_Aq_A,\qquad f'=i_{A'}q_{A'},
\]
where \(q_A,q_{A'}\) are epimorphisms and \(i_A,i_{A'}\) are the canonical monomorphisms into \(B,B'\), respectively. Then
\[
\alpha i_Aq_A=\alpha f=f'\sF(\alpha)=i_{A'}q_{A'}\sF(\alpha).
\]
Composing with \(\rho_{A'}:B'\to\Coker f'\), we obtain
\[
\rho_{A'}\alpha i_Aq_A=\rho_{A'}i_{A'}q_{A'}\sF(\alpha)=0.
\]
Since \(q_A\) is epic, we get
\[
\rho_{A'}\alpha i_A=0.
\]
Thus \(\alpha i_A:\Img f\to B'\) factors through the kernel of \(\rho_{A'}\). Since
\(i_{A'}:\Img f'\to B'\) is the kernel of \(\rho_{A'}\), the universal property of kernels, as expressed by the diagram
\[
\begin{array}{ccccc}
&&\Img f&&\\
&&{\scriptstyle \alpha i_A}\downarrow&&\\
\Img f'&\xrightarrow{i_{A'}}&B'&\xrightarrow{\rho_{A'}}&\Coker f',
\end{array}
\]
gives a unique morphism
$
\sI(\alpha):\Img f\to\Img f'
$
such that
$
i_{A'}\sI(\alpha)=\alpha i_A.
$

We next prove that \(\sI(\alpha)\) is a morphism in \(\cA\), namely
\[
\sI(\alpha)g_A=g_{A'}\sF(\sI(\alpha)).
\]
It is enough to verify this after composing with the monomorphism \(i_{A'}\). Using \(i_{A'}\sI(\alpha)=\alpha i_A\), the identities \(i_Ag_A=f\sF(i_A)\) and \(i_{A'}g_{A'}=f'\sF(i_{A'})\), and the equality \(\alpha f=f'\sF(\alpha)\), we get
\[
\begin{aligned}
i_{A'}\sI(\alpha)g_A
&=\alpha i_Ag_A\\
&=\alpha f\sF(i_A)\\
&=f'\sF(\alpha)\sF(i_A)\\
&=f'\sF(\alpha i_A)\\
&=f'\sF(i_{A'}\sI(\alpha))\\
&=f'\sF(i_{A'})\sF(\sI(\alpha))\\
&=i_{A'}g_{A'}\sF(\sI(\alpha)).
\end{aligned}
\]
Since \(i_{A'}\) is monic, we obtain
\[
\sI(\alpha)g_A=g_{A'}\sF(\sI(\alpha)).
\]
Thus \(\sI(\alpha):(\Img f,g_A)\to(\Img f',g_{A'})\) is a morphism in \(\cA\).

The identity law follows from uniqueness: both \(\sI(1_A)\) and \(1_{\sI(A)}\) become \(i_A\) after composing with \(i_A:\Img f\to B\). If \(\alpha:A\to A'\) and \(\beta:A'\to A''\) are composable, then both \(\sI(\beta\alpha)\) and \(\sI(\beta)\sI(\alpha)\) become \(\beta\alpha i_A\) after composing with \(i_{A''}\); hence they are equal. Thus \(\sI\) is a functor. Similarly, for morphisms \(\alpha,\beta:A\to A'\), both \(\sI(\alpha+\beta)\) and \(\sI(\alpha)+\sI(\beta)\) become \((\alpha+\beta)i_A\) after composing with \(i_{A'}\), so \(\sI\) is additive.

Finally, the equalities
\[
i_{A'}\sI(\alpha)=\alpha i_A,\qquad \sC(\alpha)\rho_A=\rho_{A'}\alpha
\]
show the naturality of \(\iota\) and \(\rho\). The exactness in \(\cA\) is detected by the exactness of the underlying sequence in \(\cB\). Hence
\[
0\to\sI\xrightarrow{\iota}1_{\cA}\xrightarrow{\rho}\sZ\sC\to0
\]
is a natural short exact sequence of functors.
\end{proof}

\begin{definition}\label{def:eta-Loewy}
For \(A\in\cA\), the descending chain
\[
A=\sI^0(A)\supseteq\sI(A)\supseteq\sI^2(A)\supseteq\cdots
\]
is called the \(\eta\)-Loewy filtration of \(A\). The object \(A\) is
called image-nilpotent if \(\sI^d(A)=0\) for some \(d\geq0\). Its
\(\eta\)-Loewy length is
\[
\ellength(A)=\min\{d\geq0\mid\sI^d(A)=0\}.
\]
If \(A\) is not image-nilpotent, we set \(\ellength(A)=\infty\). Let $\mathcal N_\eta=\{A\in\cA\mid\ellength(A)<\infty\}$
be the full subcategory of image-nilpotent objects. The extension
\(\cA=\cB\ltimes_\eta\sF\) is called image-nilpotent if every object
of \(\cA\) is image-nilpotent, equivalently, if $\mathcal N_\eta=\cA.$ The extension \(\cA\) is called uniformly image-nilpotent of index
at most \(d\) if $\sI^d(A)=0$ for every \(A\in\cA\).
\end{definition}

\begin{theorem}\label{thm:power-image}
Let $A=(B,f)\in\cA$. For every $r\geq1$, the underlying object of $\sI^r(A)$ is canonically isomorphic to $\Img f^{[r]}$. Consequently,
\[
\sI^r(A)=0\quad\Longleftrightarrow\quad f^{[r]}=0.
\]
Moreover, for every $r\geq0$ there is a natural short exact sequence
\[
0\longrightarrow\sI^{r+1}(A)\longrightarrow\sI^r(A)\longrightarrow\sZ(C_r(A))\longrightarrow0,
\]
where $C_r(A)=\sC(\sI^r(A))$.
\end{theorem}

\begin{proof}
The assertion is immediate for \(r=1\). Assume that the underlying object of \(\sI^r(A)\) is canonically isomorphic to
$
B_r=\Img f^{[r]}.
$
Let
\[
\sF^r(B)\xrightarrow{q_r}B_r\xrightarrow{i_r}B
\]
be the image factorization of \(f^{[r]}\). The canonical morphism \(\sI^r(A)\to A\) has underlying morphism \(i_r\), hence, if
$
g_r:\sF(B_r)\to B_r
$
denotes the structure morphism of \(\sI^r(A)\), then
$
i_rg_r=f\sF(i_r).
$
Let
\[
\sF(B_r)\xrightarrow{p_r}\Img g_r\xrightarrow{j_r}B_r
\]
be the image factorization of \(g_r\). By Lemma \ref{lem:iterated-action},
\[
\begin{aligned}
f^{[r+1]}
&=f\sF(f^{[r]})
=f\sF(i_rq_r)\\
&=f\sF(i_r)\sF(q_r)
=i_rg_r\sF(q_r)
=i_rj_rp_r\sF(q_r).
\end{aligned}
\]
Thus \(f^{[r+1]}\) factors as
$
\sF^{r+1}(B)\xrightarrow{p_r\sF(q_r)}
\Img g_r\xrightarrow{i_rj_r}B.
$
Since \(q_r\) is epic and \(\sF\) is right exact, \(\sF(q_r)\) is epic, hence \(p_r\sF(q_r)\) is epic. Moreover, \(i_rj_r\) is monic. Therefore the displayed factorization is an image factorization of \(f^{[r+1]}\), and consequently
$
\Img f^{[r+1]}\cong\Img g_r.
$
By definition, \(\Img g_r\) is the underlying object of
$
\sI(\sI^r(A))=\sI^{r+1}(A).
$
This completes the induction.

The first assertion implies
\[
\sU(\sI^r(A))\cong\Img f^{[r]}.
\]
Since \(\sU\) detects zero objects and a morphism in an abelian category is zero if and only if its image is zero, we obtain
\[
\sI^r(A)=0\quad\Longleftrightarrow\quad f^{[r]}=0.
\]

Finally, applying Proposition \ref{prop:image-functor} to \(\sI^r(A)\) gives
\[
0\to\sI(\sI^r(A))\to\sI^r(A)\to\sZ\sC(\sI^r(A))\to0.
\]
Since
\[
\sI(\sI^r(A))=\sI^{r+1}(A),\qquad C_r(A)=\sC(\sI^r(A)),
\]
this is precisely
\[
0\to\sI^{r+1}(A)\to\sI^r(A)\to\sZ(C_r(A))\to0.
\]
\end{proof}

The next result identifies uniform image-nilpotence with nilpotence of the multiplication $\eta$ itself.

\begin{theorem}\label{thm:nilpotence-criterion}
Let $d\geq1$. The following conditions are equivalent:
\begin{enumerate}
\item $\cA=\cB\ltimes_{\eta}\sF$ is uniformly image-nilpotent of index at most $d$;
\item $\eta^{[d]}:\sF^d\to\sF$ is the zero natural transformation.
\end{enumerate}
In particular, every right trivial extension is uniformly image-nilpotent of index at most two.
\end{theorem}

\begin{proof}
Assume $\eta^{[d]}=0$. For $A=(B,f)$, Lemma \ref{lem:iterated-action} gives $f^{[d]}=f\eta_B^{[d]}=0$, and Theorem \ref{thm:power-image} yields $\sI^d(A)=0$.

Conversely, suppose $\sI^d(A)=0$ for every $A$. Fix $B\in\cB$ and consider
\[
\sT(B)=\left(B\oplus\sF(B),t_B\right),\qquad
t_B=\begin{pmatrix}0&0\\1&\eta_B\end{pmatrix}.
\]
We claim that, under the canonical decomposition
\[
\sF^r(B\oplus\sF(B))\cong\sF^r(B)\oplus\sF^{r+1}(B),
\]
one has
\begin{equation}\label{eq:power-of-t}
t_B^{[r]}=\begin{pmatrix}0&0\\ \eta_B^{[r]}&\eta_B^{[r+1]}\end{pmatrix}.
\end{equation}
For $r=1$, this is the definition of $t_B$. If \eqref{eq:power-of-t} holds for $r$, then
\[
\sF(t_B^{[r]})=\begin{pmatrix}0&0\\ \sF(\eta_B^{[r]})&\sF(\eta_B^{[r+1]})\end{pmatrix},
\]
and therefore
\[
t_B^{[r+1]}=t_B\sF(t_B^{[r]})=\begin{pmatrix}0&0\\ \eta_B\sF(\eta_B^{[r]})&\eta_B\sF(\eta_B^{[r+1]})\end{pmatrix}
=\begin{pmatrix}0&0\\ \eta_B^{[r+1]}&\eta_B^{[r+2]}\end{pmatrix}.
\]
Thus \eqref{eq:power-of-t} holds for all $r$. Since $\sI^d(\sT(B))=0$, Theorem \ref{thm:power-image} gives $t_B^{[d]}=0$, and its lower-left component is $\eta_B^{[d]}$. Hence $\eta_B^{[d]}=0$ for every $B$, so $\eta^{[d]}=0$.

If $\eta=0$, then $\eta^{[2]}=0$, and the last assertion follows.
\end{proof}

\begin{remark}\label{rem:Hu-trivial-index}
The last assertion recovers, without assuming $\sF^2=0$, the two-step decomposition implicit in \cite[Lemma 3.14]{Hu}: for $\eta=0$, the first image object has zero structure morphism. Theorem \ref{thm:nilpotence-criterion} shows that this phenomenon is the degree-two instance of nilpotence of the iterated multiplication.
\end{remark}

\begin{definition}\label{def:Serre-subcategory}
A full subcategory \(\mathcal S\) of the abelian category \(\cA\) is called a Serre subcategory if it is closed under subobjects, quotient objects and extensions. 
\end{definition}

\begin{theorem}\label{thm:Serre-nilpotent}
Let
\[
0\longrightarrow A'\xrightarrow{u}A\xrightarrow{v}A''\longrightarrow0
\]
be a short exact sequence in $\cA$. Then
\[
\ellength(A'),\ellength(A'')\leq\ellength(A)
\]
and
\[
\ellength(A)\leq\ellength(A')+\ellength(A'').
\]
Consequently, the full subcategory $\cN_{\eta}$ of image-nilpotent objects is a Serre subcategory of $\cA$.
\end{theorem}

\begin{proof}  Write $A'=(B',f')$, $A=(B,f)$ and $A''=(B'',f'')$. Since $u$ and $v$ are morphisms in $\cA$, induction gives
\[
u f'^{[r]}=f^{[r]}\sF^r(u),\qquad vf^{[r]}=f''^{[r]}\sF^r(v)
\]
for every $r\geq1$. If $f^{[m]}=0$, then $u f'^{[m]}=0$, and $u$ is monic, so $f'^{[m]}=0$. Also $f''^{[m]}\sF^m(v)=0$, the morphism $\sF^m(v)$ is epic because $\sF$ is right exact, hence $f''^{[m]}=0$. The first inequalities follow from Theorem \ref{thm:power-image}.

Set \(p=\ellength(A')\) and \(q=\ellength(A'')\). If \(p=\infty\) or \(q=\infty\), the second inequality is immediate.
Hence assume that \(p,q<\infty\). If \(p=0\), then \(A'=0\) and \(A\cong A''\); if \(q=0\), then \(A''=0\) and \(A\cong A'\). In either case the required inequality is immediate. Assume that \(p,q\ge1\).
Since $f''^{[q]}=0$, one has $vf^{[q]}=0$, and therefore $f^{[q]}$ factors through $u$:
\[
f^{[q]}=uh
\]
for some $h:\sF^q(B)\to B'$. Lemma \ref{lem:iterated-action} gives
\[
f^{[p+q]}=f^{[p]}\sF^p(f^{[q]})=f^{[p]}\sF^p(u)\sF^p(h)=u f'^{[p]}\sF^p(h)=0.
\]
Thus
\[
\ellength(A)\leq p+q=\ellength(A')+\ellength(A'').
\]
By Definition \ref{def:Serre-subcategory}, \(\cN_\eta\) is a Serre subcategory of \(\cA\) .
\end{proof}

For classes \(\cD,\cQ\subseteq\cA\), let
\(\operatorname{App}(\cD,\cQ)\) denote the class of all objects \(A\in\cA\)
for which there exists a short exact sequence
\[
0\to A\to D\to Q\to0
\]
with \(D\in\cD\) and \(Q\in\cQ\).

\begin{lemma}\label{lem:relative-horseshoe}
Let $\cD$ and $\cQ$ be extension-closed classes in an abelian category $\cA$ and assume
\[
\Ext^2_{\cA}(\cQ,\cD)=0.
\]
If $0\to A'\to A\to A''\to0$ is exact and $A',A''\in\operatorname{App}(\cD,\cQ)$, then $A\in\operatorname{App}(\cD,\cQ)$.
\end{lemma}

\begin{proof}
This is the standard horseshoe argument for relative preenvelopes, compare \cite[Theorem 3.1]{AkinciAlizade}, \cite[Proposition 2.6]{Hu} and \cite[Lemma 2.8]{LongZhang}.
\end{proof}

\begin{theorem}\label{thm:filtration-reduction}
Let $\cD,\cQ\subseteq\cA$ be extension-closed and satisfy $\Ext^2_{\cA}(\cQ,\cD)=0$. Then $\operatorname{App}(\cD,\cQ)$ is extension-closed. If $A$ is image-nilpotent and every factor $\sZ(C_r(A))$ in its $\eta$-Loewy filtration belongs to $\operatorname{App}(\cD,\cQ)$, then $A\in\operatorname{App}(\cD,\cQ)$.

More precisely, suppose $\ellength(A)=m$ and choose exact sequences
\[
0\to\sZ(C_r(A))\to D_r^{\circ}\to Q_r^{\circ}\to0,
\quad D_r^{\circ}\in\cD,
\quad Q_r^{\circ}\in\cQ,
\quad0\leq r<m.
\]
Then there is an exact sequence
\[
0\to A\to D_A\to Q_A\to0
\]
with filtrations
\[
0=D_m\subseteq D_{m-1}\subseteq\cdots\subseteq D_0=D_A,
\qquad
0=Q_m\subseteq Q_{m-1}\subseteq\cdots\subseteq Q_0=Q_A
\]
such that
\[
D_r/D_{r+1}\cong D_r^{\circ},
\qquad
Q_r/Q_{r+1}\cong Q_r^{\circ}
\]
for $0\leq r<m$.
\end{theorem}

\begin{proof} If \(A=0\), take \(D_A=Q_A=0\). Hence assume that
\(A\neq0\), so \(m\geq1\). Extension-closedness of $\operatorname{App}(\cD,\cQ)$ is Lemma \ref{lem:relative-horseshoe}. By Theorem \ref{thm:power-image}, the $\eta$-Loewy filtration of $A$ consists of exact sequences
\begin{equation}\label{eq:Loewy-step}
0\longrightarrow\sI^{r+1}(A)\longrightarrow\sI^r(A)\longrightarrow\sZ(C_r(A))\longrightarrow0,
\quad0\leq r<m.
\end{equation}
We construct the asserted data by descending induction. Since $\sI^m(A)=0$, the sequence for $r=m-1$ identifies $\sI^{m-1}(A)$ with $\sZ(C_{m-1}(A))$; take
\[
D_{m-1}=D_{m-1}^{\circ},
\qquad
Q_{m-1}=Q_{m-1}^{\circ}.
\]
Assume inductively that there is an exact sequence
\[
0\longrightarrow \sI^{r+1}(A)\longrightarrow D_{r+1}\longrightarrow Q_{r+1}\longrightarrow0,
\]
and recall that, by hypothesis, we have chosen an exact sequence
\[
0\longrightarrow \sZ(C_r(A))\longrightarrow D_r^{\circ}\longrightarrow Q_r^{\circ}\longrightarrow0.
\]
Together with the Loewy step
\[
0\longrightarrow \sI^{r+1}(A)\longrightarrow \sI^r(A)\longrightarrow \sZ(C_r(A))\longrightarrow0,
\]
Lemma \ref{lem:relative-horseshoe} yields a commutative diagram with exact rows and columns
\[
\begin{array}{ccccccccc}
&&0&&0&&0&&\\
&&\downarrow&&\downarrow&&\downarrow&&\\
0&\to&\sI^{r+1}(A)&\to&\sI^r(A)&\to&\sZ(C_r(A))&\to&0\\
&&\downarrow&&\downarrow&&\downarrow&&\\
0&\to&D_{r+1}&\to&D_r&\to&D_r^{\circ}&\to&0\\
&&\downarrow&&\downarrow&&\downarrow&&\\
0&\to&Q_{r+1}&\to&Q_r&\to&Q_r^{\circ}&\to&0\\
&&\downarrow&&\downarrow&&\downarrow&&\\
&&0&&0&&0&&
\end{array}
\]
The middle column is therefore an exact sequence
\[
0\longrightarrow \sI^r(A)\longrightarrow D_r\longrightarrow Q_r\longrightarrow0,
\]
while the second and third rows are exact sequences
\[
0\longrightarrow D_{r+1}\longrightarrow D_r\longrightarrow D_r^{\circ}\longrightarrow0,
\]
and
\[
0\longrightarrow Q_{r+1}\longrightarrow Q_r\longrightarrow Q_r^{\circ}\longrightarrow0.
\]
In particular,
\[
D_r/D_{r+1}\cong D_r^{\circ},
\qquad
Q_r/Q_{r+1}\cong Q_r^{\circ}.
\]
This completes the inductive step. Taking \(r=0\) yields the exact sequence
\[
0\to A=\sI^0(A)\to D_0\to Q_0\to0.
\]
Set \(D_A=D_0\), \(Q_A=Q_0\), and \(D_m=Q_m=0\). For every \(0\leq r<m\), the induction produces short exact sequences
\[
0\to D_{r+1}\to D_r\to D_r^{\circ}\to0,
\qquad
0\to Q_{r+1}\to Q_r\to Q_r^{\circ}\to0.
\]
Thus the monomorphisms in these sequences identify \(D_{r+1}\) and \(Q_{r+1}\) as subobjects of \(D_r\) and \(Q_r\), respectively, and give filtrations
\[
0=D_m\subseteq D_{m-1}\subseteq\cdots\subseteq D_0=D_A,
\]
\[
0=Q_m\subseteq Q_{m-1}\subseteq\cdots\subseteq Q_0=Q_A.
\]
Moreover, the corresponding cokernels give
\[
D_r/D_{r+1}\cong D_r^{\circ},
\qquad
Q_r/Q_{r+1}\cong Q_r^{\circ}
\]
for every \(0\leq r<m\). Since \(\cD\) and \(\cQ\) are extension-closed and all \(D_r^{\circ}\in\cD\), \(Q_r^{\circ}\in\cQ\), induction also yields
\[
D_A\in\cD,\qquad Q_A\in\cQ.
\]
This completes the proof.
\end{proof}

\begin{corollary}\label{cor:reduction-Z}
Assume that $\cA$ is image-nilpotent. Under the hypotheses of Theorem \ref{thm:filtration-reduction}, the following are equivalent:
\begin{enumerate}
\item $\operatorname{App}(\cD,\cQ)=\cA$;
\item $\sZ(B)\in\operatorname{App}(\cD,\cQ)$ for every $B\in\cB$.
\end{enumerate}
\end{corollary}

\begin{proof}
The implication \((1)\Rightarrow(2)\) is immediate, since \(\sZ(B)\) is an object of \(\cA\) for every \(B\in\cB\). Conversely, let \(A\in\cA\). Since \(\cA\) is image-nilpotent, \(A\) has finite \(\eta\)-Loewy length, take \(\ellength(A)=m\). By Theorem \ref{thm:power-image}, there are short exact sequences
\[
0\to\sI^{r+1}(A)\to\sI^r(A)\to\sZ(C_r(A))\to0,\qquad0\leq r<m,
\]
with \(\sI^m(A)=0\). For every \(r\), the object \(C_r(A)\) belongs to \(\cB\); hence condition (2) gives
\[
\sZ(C_r(A))\in\operatorname{App}(\cD,\cQ).
\]
Thus every factor of the finite \(\eta\)-Loewy filtration of \(A\) belongs to \(\operatorname{App}(\cD,\cQ)\). Theorem \ref{thm:filtration-reduction} therefore yields
\[
A\in\operatorname{App}(\cD,\cQ).
\]
Therefore, \(\operatorname{App}(\cD,\cQ)=\cA\).
\end{proof}
\section{Higher cotorsion pairs and completeness}
We now combine the filtration theory with higher Ext-comparison.

\begin{lemma}\label{lem:Hu-Ext}
Let $X\in\cB$. Then $L_i\sF(X)\cong\sU L_i\sT(X)$ for every $i\geq1$. If $L_i\sF(X)=0$ for $1\leq i\leq m$, then for every $(Y,g)\in\cA$ there are natural isomorphisms
\[
\Ext^i_{\cA}(\sT(X),(Y,g))\cong\Ext^i_{\cB}(X,Y),
\qquad1\leq i\leq m,
\]
where  $\{L_i(-)\}_{i\in\mathbb{Z}}$ are the left derived functors.
\end{lemma}

\begin{proof}
See \cite[Lemma 3.1]{Hu}.
\end{proof}

\begin{lemma}\label{lem:induced-orthogonality}
Let $\cX\subseteq\cB$ and assume $L_i\sF(\cX)=0$ for $1\leq i\leq m$. Then
\[
\sT(\cX)^{\perp_{[1,m]}}=\sU^{-1}(\cX^{\perp_{[1,m]}}).
\]
For $m=1$, this is the orthogonality formula \cite[Lemma 3.6]{Hu}.
\end{lemma}

\begin{proof}
For $(Y,g)\in\cA$, Lemma \ref{lem:Hu-Ext} gives
\[
\Ext^i_{\cA}(\sT(X),(Y,g))\cong\Ext^i_{\cB}(X,Y),
\qquad X\in\cX,
\quad1\leq i\leq m.
\]
The asserted equality follows directly from the definitions.
\end{proof}

\begin{lemma}\label{lem:dimension-shifting}
Let $\cE,\cD$ be classes in an abelian category $\cC$. If $\Ext^i_{\cC}(\cE,\cD)=0, 1\leq i\leq n+1,$ then
$
\Ext^2_{\cC}(\cE^{\vee}_{n-1},\cD)=0.
$
Moreover, if $\Ext^i_{\cC}(\cE,\cD)=0$ for all $i\geq1$, then
\[
\Ext^i_{\cC}(\cE^{\vee}_{n-1},\cD)=0
\]
for all $i\geq1$.
\end{lemma}

\begin{proof} For \(n=1\), one has \(\mathcal E^\vee_0=\mathcal E\), and the
assertion follows directly. Assume that \(n\geq2\).  Let $Q\in\cE^{\vee}_{n-1}$ and choose
\[
0\longrightarrow Q\longrightarrow E^0\longrightarrow E^1\longrightarrow\cdots\longrightarrow E^{n-1}\longrightarrow0,
\quad E^j\in\cE.
\]
Set $K^0=Q$ and, for $1\leq j\leq n-1$, let $K^j$ be the image of $E^{j-1}\to E^j$, thus $K^{n-1}=E^{n-1}$. The short exact sequences
\[
0\longrightarrow K^j\longrightarrow E^j\longrightarrow K^{j+1}\longrightarrow0,
\qquad0\leq j\leq n-2,
\]
and the long exact Ext sequences give isomorphisms
\[
\Ext^{j+2}_{\cC}(K^j,D)\cong\Ext^{j+3}_{\cC}(K^{j+1},D),
\qquad0\leq j\leq n-2,
\]
for every $D\in\cD$. Hence
\[
\Ext^2_{\cC}(Q,D)\cong\Ext^{n+1}_{\cC}(E^{n-1},D)=0.
\]
The all-degree assertion follows by the same iteration beginning in an arbitrary positive degree.
\end{proof}

\begin{lemma}\label{lem:T-coresolutions}
If $L_1\sF(\cX^{\vee}_{n-1})=0$, then
\[
\sT(\cX^{\vee}_{n-1})\subseteq\sT(\cX)^{\vee}_{n-1}.
\]
\end{lemma}

\begin{proof}
Let $M\in\cX^{\vee}_{n-1}$ and choose
\[
0\longrightarrow M\longrightarrow X^0\longrightarrow X^1\longrightarrow\cdots\longrightarrow X^{n-1}\longrightarrow0,
\quad X^j\in\cX.
\]
Put $K^0=M$, $K^n=0$, and decompose this sequence into
\[
0\longrightarrow K^j\longrightarrow X^j\longrightarrow K^{j+1}\longrightarrow0,
\qquad0\leq j\leq n-1.
\]
Note that every $K^j$ belongs to $\cX^{\vee}_{n-1}$, hence $L_1\sF(K^{j+1})=0$, and Lemma \ref{lem:Hu-Ext} gives $\sU L_1\sT(K^{j+1})=0$. Since \(\sU(X,f)=X\), $L_1\sT(K^{j+1})=0$. Applying the right exact functor $\sT$, we obtain a series of short exact sequences
\[
0\longrightarrow\sT(K^j)\longrightarrow\sT(X^j)\longrightarrow\sT(K^{j+1})\longrightarrow0.
\]
Splicing yields a $\sT(\cX)$-coresolution of $\sT(M)$ of length at most $n-1$.
\end{proof}

The following lifting is the basic local construction. For $\eta=0$, it reduces to the object used in the proof of completeness \cite[Theorem 3.13]{Hu}, the lower-right entry $\eta_{X_B}$ is what makes the construction valid for a non-trivial $\eta$-extension.

\begin{lemma}\label{lem:basic-lifting}
Let $(\cX,\cY)$ be a right $n$-cotorsion pair in $\cB$. Assume
\[
L_1\sF(\cX^{\vee}_{n-1})=0,
\qquad
\sF(\cX^{\vee}_{n-1})\subseteq\cY.
\]
For every $B\in\cB$, there is a short exact sequence
\[
0\longrightarrow\sZ(B)\xrightarrow{u_B}E_B\xrightarrow{v_B}\sT(X_B)\longrightarrow0
\]
with
\[
E_B\in\sU^{-1}(\cY),
\qquad
\sT(X_B)\in\sT(\cX)^{\vee}_{n-1}.
\]
\end{lemma}

\begin{proof}
By Definition \ref{def:n-cotorsion-pairs}, there is an exact sequence
\[
0\longrightarrow B\xrightarrow{i}Y_B\xrightarrow{\pi}X_B\longrightarrow0,
\quad Y_B\in\cY,
\quad X_B\in\cX^{\vee}_{n-1}.
\]
Consider
\[
E_B=(Y_B\oplus\sF(X_B),\beta_B),
\qquad
\beta_B=\begin{pmatrix}0&0\\ \sF(\pi)&\eta_{X_B}\end{pmatrix}.
\]
We first verify $E_B$ is an object of $\cA$. That is,
\[
\beta_B\sF(\beta_B)=\beta_B\eta_{Y_B\oplus\sF(X_B)}.
\]

Under the canonical decompositions
\[
\sF(Y_B\oplus\sF(X_B))
\cong\sF(Y_B)\oplus\sF^2(X_B)
\]
and
\[
\sF^2(Y_B\oplus\sF(X_B))
\cong\sF^2(Y_B)\oplus\sF^3(X_B),
\]
we have
\[
\beta_B=
\begin{pmatrix}
0&0\\
\sF(\pi)&\eta_{X_B}
\end{pmatrix},
\qquad
\sF(\beta_B)=
\begin{pmatrix}
0&0\\
\sF^2(\pi)&\sF(\eta_{X_B})
\end{pmatrix}.
\]
Consequently,
\[
\beta_B\sF(\beta_B)=
\begin{pmatrix}
0&0\\
\eta_{X_B}\sF^2(\pi)&
\eta_{X_B}\sF(\eta_{X_B})
\end{pmatrix}.
\]
Since \(\sF\) and \(\eta\) are additive, the morphism
\(\eta_{Y_B\oplus\sF(X_B)}\) is represented by
\[
\eta_{Y_B\oplus\sF(X_B)}=
\begin{pmatrix}
\eta_{Y_B}&0\\
0&\eta_{\sF(X_B)}
\end{pmatrix}.
\]
Hence
\[
\beta_B\eta_{Y_B\oplus\sF(X_B)}=
\begin{pmatrix}
0&0\\
\sF(\pi)\eta_{Y_B}&
\eta_{X_B}\eta_{\sF(X_B)}
\end{pmatrix}.
\]
By the naturality and associativity of $\eta$, the diagrams
\[
\begin{array}{cc}
\begin{array}{ccc}
\sF^2(Y_B)&\xrightarrow{\eta_{Y_B}}&\sF(Y_B)\\
{\scriptstyle\sF^2(\pi)}\downarrow&&{\scriptstyle\sF(\pi)}\downarrow\\
\sF^2(X_B)&\xrightarrow{\eta_{X_B}}&\sF(X_B)
\end{array}
&
\begin{array}{ccc}
\sF^3(X_B)&\xrightarrow{\sF(\eta_{X_B})}&\sF^2(X_B)\\
{\scriptstyle\eta_{\sF(X_B)}}\downarrow&&{\scriptstyle\eta_{X_B}}\downarrow\\
\sF^2(X_B)&\xrightarrow{\eta_{X_B}}&\sF(X_B)
\end{array}
\end{array}
\]
commute, and hence
$
\beta_B\sF(\beta_B)
=\beta_B\eta_{Y_B\oplus\sF(X_B)}.
$
Therefore \(E_B=(Y_B\oplus\sF(X_B),\beta_B)\) is an object of
\(\cA=\cB\ltimes_\eta\sF\).

Next, define
\[
{u}_B=\binom{i}{0}:B\to Y_B\oplus\sF(X_B)
\]
and
\[
{v}_B=
\begin{pmatrix}
\pi&0\\
0&1_{\sF(X_B)}
\end{pmatrix}:
Y_B\oplus\sF(X_B)\to X_B\oplus\sF(X_B)
\]
be morphisms in \(\cB\). We show that they satisfy the compatibility relations required to define morphisms
\[
u_B:\sZ(B)\to E_B,\qquad
v_B:E_B\to\sT(X_B)
\]
in \(\cA\). That is,
$
{u}_B0=\beta_B\sF({u}_B),
$
and
$
{v}_B\beta_B=t_{X_B}\sF({v}_B).
$

Since
\[
\sF({u}_B)=\binom{\sF(i)}{0},
\]
we have
\[
\beta_B\sF({u}_B)
=
\begin{pmatrix}
0&0\\
\sF(\pi)&\eta_{X_B}
\end{pmatrix}
\binom{\sF(i)}{0}
=
\binom{0}{\sF(\pi i)}
=0.
\]
Thus \(u_B:\sZ(B)\to E_B\) is a morphism in \(\cA\).

Recall that
\[
\sT(X_B)=\bigl(X_B\oplus\sF(X_B),t_{X_B}\bigr),
\]
where
\[
t_{X_B}=
\begin{pmatrix}
0&0\\
1_{\sF(X_B)}&\eta_{X_B}
\end{pmatrix}:
\sF(X_B)\oplus\sF^2(X_B)\to X_B\oplus\sF(X_B).
\]
Under the canonical decompositions,
\[
\sF({v}_B)=
\begin{pmatrix}
\sF(\pi)&0\\
0&1_{\sF^2(X_B)}
\end{pmatrix}.
\]
Therefore
\[
{v}_B\beta_B
=
\begin{pmatrix}
\pi&0\\
0&1_{\sF(X_B)}
\end{pmatrix}
\begin{pmatrix}
0&0\\
\sF(\pi)&\eta_{X_B}
\end{pmatrix}
=
\begin{pmatrix}
0&0\\
\sF(\pi)&\eta_{X_B}
\end{pmatrix},
\]
while
\[
t_{X_B}\sF({v}_B)
=
\begin{pmatrix}
0&0\\
1_{\sF(X_B)}&\eta_{X_B}
\end{pmatrix}
\begin{pmatrix}
\sF(\pi)&0\\
0&1_{\sF^2(X_B)}
\end{pmatrix}
=
\begin{pmatrix}
0&0\\
\sF(\pi)&\eta_{X_B}
\end{pmatrix}.
\]
Hence \(v_B:E_B\to\sT(X_B)\) is a morphism in \(\cA\).

It remains to prove exactness. 
The two exact sequences
\[
0\to B\xrightarrow{i}Y_B\xrightarrow{\pi}X_B\to0
\]
and
\[
0\to0\to\sF(X_B)
\xrightarrow{1_{\sF(X_B)}}\sF(X_B)\to0,
\]

 give the  exact sequence
\[
0\to B\xrightarrow{\binom{i}{0}}
Y_B\oplus\sF(X_B)
\xrightarrow{\left(\begin{smallmatrix}
		\pi&0\\
		0&1_{\sF(X_B)}
	\end{smallmatrix}\right)}
X_B\oplus\sF(X_B)\to0.
\]
in \(\cB\). By Lemma \ref{lem:1},
\[
0\to\sZ(B)\xrightarrow{u_B}E_B
\xrightarrow{v_B}\sT(X_B)\to0
\]
is a short exact sequence in \(\cA\).

By Lemma \ref{lem:T-coresolutions}, $\sT(X_B)\in\sT(\cX)^{\vee}_{n-1}$. Moreover,
\[
\sU(E_B)=Y_B\oplus\sF(X_B)\in\cY,
\]
because $X_B\in\cX^{\vee}_{n-1}$, $\sF(\cX^{\vee}_{n-1})\subseteq\cY$, and $\cY=\cX^{\perp_{[1,n]}}$ is closed under finite direct sums by Proposition \ref{prop:HMP}. Thus $E_B\in\sU^{-1}(\cY)$. This proves the lemma.
\end{proof}

\begin{remark}\label{rem:basic-lifting}
The construction in Lemma \ref{lem:basic-lifting} is forced by the induction functor. Starting from
\[
0\to B\xrightarrow{i}Y_B\xrightarrow{\pi}X_B\to0,
\]
we choose \(\sT(X_B)\) as the right-hand term because Lemma \ref{lem:T-coresolutions} transfers finite \(\cX\)-coresolutions to finite \(\sT(\cX)\)-coresolutions. When \(\eta=0\), one has
\[
t_{X_B}=\begin{pmatrix}0&0\\1&0\end{pmatrix},
\qquad
\beta_B=\begin{pmatrix}0&0\\ \sF(\pi)&0\end{pmatrix}.
\]
Thus the sequence in Lemma \ref{lem:basic-lifting} becomes
\[
0\to\sZ(B)\xrightarrow{\binom{i}{0}}
\left(Y_B\oplus\sF(X_B),
\begin{pmatrix}0&0\\ \sF(\pi)&0\end{pmatrix}\right)
\xrightarrow{\left(\begin{smallmatrix}\pi&0\\0&1\end{smallmatrix}\right)}
\sT(X_B)\to0,
\]
which is exactly the sequence constructed in the proof of
\cite[Theorem 3.13]{Hu}. Hence the exact sequence in Lemma \ref{lem:basic-lifting} extends this construction from right trivial extensions to general \(\eta\)-extensions.
\end{remark}

\begin{theorem}\label{thm:local-global-right-n}
Let \(\cA=\cB\ltimes_\eta\sF\) be image-nilpotent, let \(n\geq1\), and let \(\mathcal P,\mathcal R\) be classes of objects in \(\cA\). Assume that
\begin{enumerate}
\item \(\mathcal R\) is closed under extensions and direct summands;
\item \(\mathcal P^\vee_{n-1}\) is closed under extensions;
\item
\[
\Ext^i_{\cA}(\mathcal P,\mathcal R)=0,\qquad1\leq i\leq n;
\]
\item
\[
\Ext^2_{\cA}(\mathcal P^\vee_{n-1},\mathcal R)=0.
\]
\end{enumerate}
Then the following conditions are equivalent:
\begin{enumerate}
\item[(a)] \((\mathcal P,\mathcal R)\) is a right \(n\)-cotorsion pair in \(\cA\);
\item[(b)] for every \(B\in\cB\), there is a short exact sequence
\[
0\to\sZ(B)\to R_B\to P_B\to0
\]
with \(R_B\in\mathcal R\) and \(P_B\in\mathcal P^\vee_{n-1}\).
\end{enumerate}

If \(\mathcal R\) is coresolving, then condition {\rm(4)} follows from condition {\rm(3)}.
\end{theorem}

\begin{proof}
The implication \({\rm(a)}\Rightarrow{\rm(b)}\) follows immediately from the definition of a right \(n\)-cotorsion pair, applied to the objects \(\sZ(B)\).

Conversely,  let \(A\in\cA\). Since \(\cA\) is image-nilpotent, \(A\) has a finite \(\eta\)-Loewy filtration
\[
0=\sI^m(A)\subseteq\sI^{m-1}(A)\subseteq\cdots\subseteq\sI^0(A)=A
\]
with short exact sequences
\[
0\to\sI^{r+1}(A)\to\sI^r(A)\to\sZ(C_r(A))\to0,
\qquad0\leq r<m.
\]
By (b), every factor \(\sZ(C_r(A))\) belongs to
$
\operatorname{App}\bigl(\mathcal R,\mathcal P^\vee_{n-1}\bigr).
$
Conditions (1), (2) and (4) allow us to apply Theorem
\ref{thm:filtration-reduction}. Hence there is a short exact sequence
\[
0\to A\to R_A\to P_A\to0
\]
with \(R_A\in\mathcal R\) and \(P_A\in\mathcal P^\vee_{n-1}\). Hence, Definition \ref{def:n-cotorsion-pairs} shows that
$
(\mathcal P,\mathcal R)
$
is a right \(n\)-cotorsion pair in \(\cA\) .

Finally, if \(\mathcal R\) is coresolving, then  condition {\rm(3)} and Proposition
\ref{prop:HMP}(2) give
$
\Ext^1_{\cA}(\mathcal P^\vee_{n-1},\mathcal R)=0.
$
Let \(Q\in\mathcal P^\vee_{n-1}\) and \(R\in\mathcal R\). Choose a
short exact sequence
\[
0\to R\to I\to R_1\to0
\]
with \(I\) injective. Since \(\mathcal R\) is coresolving,
\(R_1\in\mathcal R\). Dimension shifting yields
\[
\Ext^2_{\cA}(Q,R)\cong\Ext^1_{\cA}(Q,R_1)=0.
\]
Hence
$
\Ext^2_{\cA}(\mathcal P^\vee_{n-1},\mathcal R)=0,
$
so condition {\rm(4)} follows from condition {\rm(3)}.
\end{proof}

\begin{corollary}\label{cor:local-global-left-n}
Let \(\cA=\cB\ltimes_\eta\sF\) be image-nilpotent, and let
\(\mathcal P,\mathcal R\subseteq\cA\). Assume that
\begin{enumerate}
\item \(\mathcal P\) is closed under extensions and direct summands;
\item \(\mathcal R^\wedge_{n-1}\) is closed under extensions;
\item
\[
\Ext^i_{\cA}(\mathcal P,\mathcal R)=0,\qquad1\leq i\leq n;
\]
\item
\[
\Ext^2_{\cA}(\mathcal P,\mathcal R^\wedge_{n-1})=0.
\]
\end{enumerate}
Then \((\mathcal P,\mathcal R)\) is a left \(n\)-cotorsion pair if and only if, for every \(B\in\cB\), there is a short exact sequence
\[
0\to R_B\to P_B\to\sZ(B)\to0
\]
with \(P_B\in\mathcal P\) and \(R_B\in\mathcal R^\wedge_{n-1}\).
\end{corollary}

\begin{proof}
The proof is dual to that of Theorem
\ref{thm:local-global-right-n} and is omitted.
\end{proof}

To simplify the formulation of the main results, we collect the homological vanishing, compatibility and extension-closure conditions into the following definition.

\begin{definition}\label{def:F-compatible}
A right $n$-cotorsion pair $(\cX,\cY)$ in $\cB$ is $\sF$-compatible if
\begin{enumerate}
\item $L_i\sF(\cX)=0$ for $1\leq i\leq n+1$;
\item $L_1\sF(\cX^{\vee}_{n-1})=0$ and $\sF(\cX^{\vee}_{n-1})\subseteq\cY$;
\item $\Ext^{n+1}_{\cB}(\cX,\cY)=0$;
\item $\sT(\cX)^{\vee}_{n-1}$ is extension-closed in $\cA$.
\end{enumerate}
It is strongly $\sF$-compatible if it is hereditary and $L_i\sF(\cX)=0$ for all $i\geq1$.
\end{definition}

\begin{remark}\label{rem:extra-degree}
Although Definition \ref{def:F-compatible} may appear somewhat cumbersome at first sight, its assumptions are in fact quite natural, as the examples in Section~\ref{sec:examples} will show.
\end{remark}

\begin{theorem}\label{thm:lift-n-cotorsion}
Let \(\cA=\cB\ltimes_\eta\sF\) be image-nilpotent, and let
\((\cX,\cY)\) be an \(\sF\)-compatible right \(n\)-cotorsion pair in
\(\cB\). Then
$
\bigl(\sT(\cX),\sU^{-1}(\cY)\bigr)
$
is a right \(n\)-cotorsion pair in \(\cA\), and
$
\sU^{-1}(\cY)=\sT(\cX)^{\perp_{[1,n]}}.
$

More precisely, let \(A\in\cA\) have \(\eta\)-Loewy length \(m\).
For every \(0\leq r<m\), choose a short exact sequence
\[
0\to C_r(A)\to Y_r\to X_r\to0,
\qquad Y_r\in\cY,\quad X_r\in\cX^\vee_{n-1},
\]
and let
$
0\to\sZ(C_r(A))\to E_r\to\sT(X_r)\to0
$
be its lift given by Lemma \ref{lem:basic-lifting}. Then there is a
short exact sequence
$
0\to A\to D_A\to Q_A\to0
$
with
$
D_A\in\sU^{-1}(\cY), Q_A\in\sT(\cX)^\vee_{n-1}.
$
Moreover, \(D_A\) and \(Q_A\) admit filtrations
\[
0=D_m\subseteq D_{m-1}\subseteq\cdots\subseteq D_0=D_A,
\]
\[
0=Q_m\subseteq Q_{m-1}\subseteq\cdots\subseteq Q_0=Q_A
\]
such that
$
D_r/D_{r+1}\cong E_r, Q_r/Q_{r+1}\cong\sT(X_r)$ for \(0\leq r<m\).
\end{theorem}

\begin{proof}
Put
$
\mathcal P=\sT(\cX), \mathcal R=\sU^{-1}(\cY).
$
By Proposition \ref{prop:HMP},
$
\cY=\cX^{\perp_{[1,n]}}.
$
Hence \(\cY\) is closed under extensions and direct summands. Since
\(\sU\) is exact and additive, \(\mathcal R\) has the same closure
properties. Moreover,
$
\mathcal P^\vee_{n-1}
=\sT(\cX)^\vee_{n-1}
$
is extension-closed by Definition \ref{def:F-compatible}(4).

Let \(X\in\cX\) and \(M=(Y,g)\in\mathcal R\). Then \(Y\in\cY\).
Definition \ref{def:F-compatible}(1) and Lemma
\ref{lem:Hu-Ext} give
\[
\Ext^i_{\cA}(\sT(X),M)\cong\Ext^i_{\cB}(X,Y),
\qquad1\leq i\leq n+1.
\]
For \(1\leq i\leq n\), the group on the right vanishes because
\((\cX,\cY)\) is a right \(n\)-cotorsion pair, while for \(i=n+1\)
it vanishes by Definition \ref{def:F-compatible}(3). Therefore
\[
\Ext^i_{\cA}(\mathcal P,\mathcal R)=0,
\qquad1\leq i\leq n+1.
\]
Lemma \ref{lem:dimension-shifting} yields
$
\Ext^2_{\cA}(\mathcal P^\vee_{n-1},\mathcal R)=0.
$
For every \(B\in\cB\), the right \(n\)-cotorsion pair
\((\cX,\cY)\) gives a short exact sequence
\[
0\to B\to Y_B\to X_B\to0,
\qquad Y_B\in\cY,\quad X_B\in\cX^\vee_{n-1}.
\]
By Definition \ref{def:F-compatible}(2), Lemma
\ref{lem:basic-lifting} gives
\[
0\to\sZ(B)\to E_B\to\sT(X_B)\to0
\]
with
$
E_B\in\mathcal R, \sT(X_B)\in\mathcal P^\vee_{n-1}.
$
Thus all hypotheses of Theorem
\ref{thm:local-global-right-n} are satisfied. Hence
\[
\bigl(\mathcal P,\mathcal R\bigr)
=
\bigl(\sT(\cX),\sU^{-1}(\cY)\bigr)
\]
is a right \(n\)-cotorsion pair in \(\cA\).

Furthermore, Definition \ref{def:F-compatible}(1), Lemma
\ref{lem:induced-orthogonality} and
\(\cY=\cX^{\perp_{[1,n]}}\) give
\[
\sT(\cX)^{\perp_{[1,n]}}
=\sU^{-1}\bigl(\cX^{\perp_{[1,n]}}\bigr)
=\sU^{-1}(\cY).
\]

If \(A=0\), take
\(D_A=Q_A=0\). Assume \(A\neq0\). For every \(0\leq r<m\), the
chosen lift satisfies
$
E_r\in\mathcal R, \sT(X_r)\in\mathcal P^\vee_{n-1}.
$
Applying Theorem \ref{thm:filtration-reduction} with
$
\mathcal D=\mathcal R, \mathcal Q=\mathcal P^\vee_{n-1},
$
and with
$
D_r^\circ=E_r, Q_r^\circ=\sT(X_r),
$
gives a short exact sequence
\[
0\to A\to D_A\to Q_A\to0
\]
and filtrations
\[
0=D_m\subseteq D_{m-1}\subseteq\cdots\subseteq D_0=D_A,
\]
\[
0=Q_m\subseteq Q_{m-1}\subseteq\cdots\subseteq Q_0=Q_A
\]
whose successive factors satisfy
\[
D_r/D_{r+1}\cong E_r,\qquad
Q_r/Q_{r+1}\cong\sT(X_r).
\]
This completes the proof.
\end{proof}

\begin{corollary}\label{cor:hereditary-lifting}
Under the hypotheses of Theorem \ref{thm:lift-n-cotorsion}, assume
further that \((\cX,\cY)\) is hereditary and
\[
L_i\sF(\cX)=0\qquad\text{for all }i\geq1.
\]
Then
\[
\Ext^i_{\cA}\bigl(\sT(\cX)^\vee_{n-1},
\sU^{-1}(\cY)\bigr)=0
\qquad\text{for all }i\geq1.
\]
Consequently,
\[
\bigl(\sT(\cX),\sU^{-1}(\cY)\bigr)
\]
is a hereditary right \(n\)-cotorsion pair in \(\cA\).
\end{corollary}

\begin{proof}
Let \(X\in\cX\) and \(M=(Y,g)\in\sU^{-1}(\cY)\). Then
\(Y\in\cY\). Since \((\cX,\cY)\) is hereditary and
\(L_j\sF(X)=0\) for all \(j\geq1\), Lemma
\ref{lem:Hu-Ext} gives, for every \(i\geq1\),
\[
\Ext^i_{\cA}(\sT(X),M)
\cong\Ext^i_{\cB}(X,Y)=0.
\]
Hence
\[
\Ext^i_{\cA}\bigl(\sT(\cX),\sU^{-1}(\cY)\bigr)=0
\]
for every \(i\geq1\). The all-degree part of Lemma
\ref{lem:dimension-shifting} now gives
\[
\Ext^i_{\cA}
\bigl(\sT(\cX)^\vee_{n-1},\sU^{-1}(\cY)\bigr)=0
\]
for every \(i\geq1\). Together with Theorem
\ref{thm:lift-n-cotorsion}, this proves the assertion.
\end{proof}

\begin{theorem}\label{thm:complete-cotorsion}
Let \(\cA=\cB\ltimes_{\eta}\sF\) be image-nilpotent, and let \((\cX,\cY)\) be a hereditary complete cotorsion pair in \(\cB\). If
\[
L_1\sF(\cX)=0,\qquad \sF(\cX)\subseteq\cY,
\]
then
\[
\bigl({}^{\perp}\sU^{-1}(\cY),\sU^{-1}(\cY)\bigr)
\]
is a hereditary complete cotorsion pair in \(\cA\).
\end{theorem}

\begin{proof}
In the general \(\eta\)-extension \(\cA\), by \cite[Theorem 3.5]{Hu}, the condition \(L_1\sF(\cX)=0\) implies that
\[
\bigl({}^{\perp}\sU^{-1}(\cY),\sU^{-1}(\cY)\bigr)
\]
is a cotorsion pair in \(\cA\). Moreover, it is hereditary because \((\cX,\cY)\) is hereditary. It remains to prove completeness.

Let \(B\in\cB\). Since \((\cX,\cY)\) is complete, there is a short exact sequence
\[
0\to B\to Y_B\to X_B\to0,\qquad Y_B\in\cY,\quad X_B\in\cX.
\]
Since \(\cX^\vee_0=\cX\), Lemma \ref{lem:basic-lifting} gives
\begin{equation}\label{eq:special-Z}
0\to\sZ(B)\to E_B\to\sT(X_B)\to0,
\qquad E_B\in\sU^{-1}(\cY).
\end{equation}
For \(M=(Y,g)\in\sU^{-1}(\cY)\), one has \(Y\in\cY\), and Lemma \ref{lem:Hu-Ext} yields
\[
\Ext^1_{\cA}(\sT(X_B),M)\cong\Ext^1_{\cB}(X_B,Y)=0.
\]
Hence \(\sT(X_B)\in{}^{\perp}\sU^{-1}(\cY)\), so \eqref{eq:special-Z} is a special \(\sU^{-1}(\cY)\)-preenvelope of \(\sZ(B)\).

Since the induced cotorsion pair is hereditary, the classes
\[
{}^{\perp}\sU^{-1}(\cY)\quad\text{and}\quad\sU^{-1}(\cY)
\]
are extension-closed and satisfy
\[
\Ext^2_{\cA}\bigl({}^{\perp}\sU^{-1}(\cY),\sU^{-1}(\cY)\bigr)=0.
\]
Let \(A\in\cA\). Its finite \(\eta\)-Loewy filtration has factors \(\sZ(C_r(A))\). Applying \eqref{eq:special-Z} to each \(C_r(A)\), and then Theorem \ref{thm:filtration-reduction} gives a short exact sequence
\[
0\to A\to D_A\to Q_A\to0,
\qquad D_A\in\sU^{-1}(\cY),\quad
Q_A\in{}^{\perp}\sU^{-1}(\cY).
\]
Thus every object of \(\cA\) admits a special
\(\sU^{-1}(\cY)\)-preenvelope. Proposition
\ref{prop:cotorsion-criteria}(1) therefore shows that the
cotorsion pair is complete.
\end{proof}

\begin{remark}\label{rem:comparison-Hu-complete}
The theorem \cite[Theorem 3.13]{Hu} is proved under the additional
assumption \(\eta=0\), in which case every object has image
length at most two.
\end{remark}

\section{Examples and applications}\label{sec:examples}
We first treat split nilpotent extensions with non-zero multiplication, the remaining examples specialize the general construction to familiar right trivial extensions.

\begin{example}\label{ex:split-nilpotent-ring}
Let \(R\) be a ring and let \(I\) be an \(R\)-\(R\)-bimodule endowed with an associative \(R\)-bimodule homomorphism
\[
\mu:I\otimes_RI\to I.
\]
Set \(\Gamma=R\oplus I\) with multiplication
\[
(r,a)(s,b)=(rs,rb+as+\mu(a\otimes b)).
\]
Then \(J=0\oplus I\) is an ideal of \(\Gamma\) and \(\Gamma/J\cong R\).

For \(r\geq1\), write \(I^{\otimes_Rr}=I\otimes_R\cdots\otimes_RI\) and define the multiplication
\[
\mu^{[r]}:I^{\otimes_Rr}\to I
\]
recursively by
\[
\mu^{[1]}=\operatorname{id}_I,\qquad \mu^{[2]}=\mu,\qquad
\mu^{[r+1]}=\mu\bigl(\mu^{[r]}\otimes_R\operatorname{id}_I)\quad(r\geq2).
\]
Since \(\mu\) is associative, one also has
\[
\mu^{[r+1]}=\mu\bigl(\operatorname{id}_I\otimes_R\mu^{[r]}\bigr).
\]
Thus \(\mu^{[r]}\) is independent of the placement of
parentheses and
\[
\mu^{[r]}(a_1\otimes\cdots\otimes a_r)=a_1a_2\cdots a_r.
\]

For \(r\geq1\), let \(J^r\) be the ideal generated by all products
\(j_1\cdots j_r\) with \(j_1,\ldots,j_r\in J\). Since
\[
(0,a_1)\cdots(0,a_r)
=\bigl(0,\mu^{[r]}(a_1\otimes\cdots\otimes a_r)\bigr),
\]
one has
$
J^r=0\oplus\Img\mu^{[r]}.
$
Assume from now on that
$
J^d=0
$
for some \(d\geq1\). Equivalently, \(\mu^{[d]}=0\).

Put \(\sF=I\otimes_R-\). Under the canonical identification
\[
\sF^2(M)
=I\otimes_R(I\otimes_RM)
\cong(I\otimes_RI)\otimes_RM,
\]
define
\[
\eta_M=\mu\otimes_R\operatorname{id}_M:
I\otimes_RI\otimes_RM\to I\otimes_RM.
\]
Thus
\[
\eta_M(a\otimes b\otimes m)
=\mu(a\otimes b)\otimes m.
\]
Naturality is immediate, and the associativity of \(\mu\) gives
\[
\eta\circ\sF\eta=\eta\circ\eta\sF.
\]
Thus \(\eta:\sF^2\to\sF\) is an associative natural transformation.

Let
\[
\iota:R\to\Gamma,\qquad r\mapsto(r,0)
\]
be the canonical ring homomorphism. If \(M\) is a left
\(\Gamma\)-module, restriction of scalars along \(\iota\) makes \(M\)
a left \(R\)-module. The action of \(J=0\oplus I\) on \(M\) defines
a map
\[
f_M:I\otimes_RM\to M,\qquad f_M(a\otimes m)=(0,a)m.
\]
This map is well defined. Indeed, for \(a\in I\), \(r\in R\) and
\(m\in M\),
\[
f_M(ar\otimes m)=(0,ar)m=(0,a)(r,0)m
=f_M(a\otimes rm),
\]
and
\[
f_M(ra\otimes m)=(0,ra)m=(r,0)(0,a)m
=r f_M(a\otimes m).
\]
Thus \(f_M\) is \(R\)-linear.

The associativity of the \(\Gamma\)-action implies the object
relation in the \(\eta\)-extension. For \(a,b\in I\) and \(m\in M\),
\[
\begin{aligned}
f_M(\operatorname{id}_I\otimes f_M)(a\otimes b\otimes m)
&=(0,a)\bigl((0,b)m\bigr)\\
&=((0,a)(0,b))m\\
&=(0,\mu(a\otimes b))m\\
&=f_M(\mu\otimes_R\operatorname{id}_M)
(a\otimes b\otimes m).
\end{aligned}
\]
Hence
\[
f_M(\operatorname{id}_I\otimes f_M)
=f_M(\mu\otimes_R\operatorname{id}_M)
=f_M\eta_M,
\]
so \((M,f_M)\) is an object of
\[
R\text{-}\Mod\ltimes_\eta(I\otimes_R-).
\]

Conversely, let \((M,f)\) be an object of this \(\eta\)-extension.
Define an action of \(\Gamma\) on \(M\) by
\[
(r,a)\cdot_f m=rm+f(a\otimes m).
\]
The element \((1_R,0)\) acts as the identity. Moreover,
\[
\begin{aligned}
(r,a)\cdot_f\bigl((s,b)\cdot_fm\bigr)
&=rsm+r f(b\otimes m)+f(a\otimes sm)
  +f\bigl(a\otimes f(b\otimes m)\bigr),
\end{aligned}
\]
whereas
\[
\begin{aligned}
\bigl((r,a)(s,b)\bigr)\cdot_fm
&=rsm+f(rb\otimes m)+f(as\otimes m)
  +f(\mu(a\otimes b)\otimes m).
\end{aligned}
\]
Since \(f\) is \(R\)-linear and balanced,
\[
f(rb\otimes m)=r f(b\otimes m),\qquad
f(as\otimes m)=f(a\otimes sm),
\]
and the object relation gives
\[
f\bigl(a\otimes f(b\otimes m)\bigr)
=f(\mu(a\otimes b)\otimes m).
\]
Thus the two expressions coincide, and \(\cdot_f\) defines a left
\(\Gamma\)-module structure on \(M\).

Let \((M,f)\) and \((N,g)\) be objects of the \(\eta\)-extension. An
\(R\)-linear map \(\alpha:M\to N\) is \(\Gamma\)-linear if and only if
it commutes with the action of \(I\), that is,
\[
\alpha f=g(\operatorname{id}_I\otimes\alpha).
\]
Indeed, this equality is equivalent to
\[
\alpha\bigl((0,a)m\bigr)=(0,a)\alpha(m)
\]
for all \(a\in I\) and \(m\in M\), \(R\)-linearity already guarantees
compatibility with the action of every \((r,0)\).

We therefore obtain mutually inverse functors
\[
\Phi:\Gamma\text{-}\Mod\to
R\text{-}\Mod\ltimes_\eta(I\otimes_R-),
\qquad
M\mapsto(M,f_M),
\]
and
\[
\Psi:R\text{-}\Mod\ltimes_\eta(I\otimes_R-)
\longrightarrow\Gamma\text{-}\Mod,
\qquad
(M,f)\longmapsto M_f,
\]
where \(M_f\) denotes the \(R\)-module \(M\) endowed with the left
\(\Gamma\)-action
\[
(r,a)\cdot_fm=rm+f(a\otimes m).
\]
On morphisms, \(\Psi\) acts as the identity on the underlying
\(R\)-linear maps.
Both functors act identically on the underlying \(R\)-modules and
\(R\)-linear maps. For a \(\Gamma\)-module \(M\),
\[
(r,a)\cdot_{f_M}m
=rm+f_M(a\otimes m)
=(r,0)m+(0,a)m
=(r,a)m,
\]
so \(\Psi\Phi(M)=M\) as a \(\Gamma\)-module. Conversely, for an
object \((M,f)\) of the \(\eta\)-extension,
\[
f_{M_f}(a\otimes m)
=(0,a)\cdot_fm
=f(a\otimes m),
\]
and hence \(\Phi\Psi(M,f)=(M,f)\). Therefore
\[
\Phi\Psi
=1_{R\text{-}\Mod\ltimes_\eta(I\otimes_R-)},
\qquad
\Psi\Phi=1_{\Gamma\text{-}\Mod}.
\]
Hence there is an exact
equivalence
\[
\Gamma\text{-}\Mod\simeq
R\text{-}\Mod\ltimes_\eta(I\otimes_R-),
\]
compare \cite{FossumGriffithReiten,BeligiannisCleft}.

For \(r\geq1\), the \(r\)-fold natural transformation
\[
\eta^{[r]}:\sF^r\to\sF
\]
has component
\[
\eta_M^{[r]}
=\mu^{[r]}\otimes_R\operatorname{id}_M:
I^{\otimes_Rr}\otimes_RM\to I\otimes_RM.
\]
By the assumption \(J^d=0\), one has \(\mu^{[d]}=0\). Since
\[
\eta_M^{[d]}
=\mu^{[d]}\otimes_R\operatorname{id}_M
\]
for every \(M\), it follows that \(\eta^{[d]}=0\). Hence Theorem
\ref{thm:nilpotence-criterion} shows that the extension is uniformly
image-nilpotent of index at most \(d\).

For a \(\Gamma\)-module \(M\), the iterated action map is
\[
f_M^{[r]}:
I^{\otimes_Rr}\otimes_RM\to M,
\]
given by
\[
f_M^{[r]}(a_1\otimes\cdots\otimes a_r\otimes m)
=(0,a_1)\cdots(0,a_r)m.
\]
Therefore
$
\operatorname{Im}f_M^{[r]}=J^rM.
$
By Theorem \ref{thm:power-image},
$
\sU\bigl(\sI^r(\Phi(M))\bigr)
\cong\operatorname{Im}f_M^{[r]}=J^rM.
$
Thus, under the equivalence \(\Phi\), the \(\eta\)-Loewy filtration
corresponds to the \(J\)-adic filtration
\[
M\supseteq JM\supseteq J^2M\supseteq\cdots\supseteq J^dM=0.
\]
\end{example}

\begin{corollary}\label{cor:split-nilpotent-cotorsion}
Continue with the notation of Example \ref{ex:split-nilpotent-ring}. Put
$
\cY_\Gamma=
\{M\in\Gamma\text{-}\Mod\mid \operatorname{Res}_R^\Gamma(M)\in\cY\}.
$
The following statements hold.
\begin{enumerate}
\item[(1)] If \((\cX,\cY)\) is an \(\sF\)-compatible right
\(n\)-cotorsion pair in \(R\text{-}\Mod\), where
\(\sF=I\otimes_R-\), then
$
\bigl(\Gamma\otimes_R\cX,\cY_\Gamma\bigr)
$
is a right \(n\)-cotorsion pair in \(\Gamma\text{-}\Mod\).

\item[(2)] If \((\cX,\cY)\) is a hereditary complete cotorsion pair
satisfying
\[
\Tor_1^R(I,\cX)=0,\qquad I\otimes_R\cX\subseteq\cY,
\]
then
$
\bigl({}^\perp\cY_\Gamma,\cY_\Gamma\bigr)
$
is a hereditary complete cotorsion pair in
\(\Gamma\text{-}\Mod\), where the left orthogonal is taken in
\(\Gamma\text{-}\Mod\).
\end{enumerate}
\end{corollary}

\begin{proof}
Let
\[
\Phi:\Gamma\text{-}\Mod\xrightarrow{\sim}
R\text{-}\Mod\ltimes_\eta(I\otimes_R-)
\]
be the exact equivalence constructed in Example
\ref{ex:split-nilpotent-ring}, and let \(\Psi\) be a quasi-inverse.
The same example shows that the \(\eta\)-extension is image-nilpotent.
Since
$
\sU\Phi=\operatorname{Res}_R^\Gamma,
$
we have
$
\Psi\bigl(\sU^{-1}(\cY)\bigr)=\cY_\Gamma.
$

For \(X\in R\text{-}\Mod\), consider the \(R\)-linear isomorphism
\[
\theta_X:\Gamma\otimes_RX\to X\oplus(I\otimes_RX),\qquad
\theta_X((r,a)\otimes x)=(rx,a\otimes x).
\]
It satisfies
$
\theta_Xf_{\Gamma\otimes_RX}
=t_X(\operatorname{id}_I\otimes_R\theta_X),
$
and hence defines an isomorphism in the \(\eta\)-extension
$
\Phi(\Gamma\otimes_RX)\cong\sT(X).
$
Therefore
$
\Psi\bigl(\sT(\cX)\bigr)=\Gamma\otimes_R\cX.
$

By Theorem \ref{thm:lift-n-cotorsion},
$
\bigl(\sT(\cX),\sU^{-1}(\cY)\bigr)
$
is a right \(n\)-cotorsion pair. Applying Lemma \ref{lem:exact-equivalence} gives
$
\bigl(\Gamma\otimes_R\cX,\cY_\Gamma\bigr)
$
as a right \(n\)-cotorsion pair in \(\Gamma\text{-}\Mod\). This proves
{\rm(1)}.

For {\rm(2)}, since \(\sF=I\otimes_R-\), one has
\[
L_1\sF(\cX)\cong\Tor_1^R(I,\cX)=0,
\qquad
\sF(\cX)=I\otimes_R\cX\subseteq\cY.
\]
Theorem \ref{thm:complete-cotorsion} therefore gives a hereditary
complete cotorsion pair
\[
\bigl({}^\perp\sU^{-1}(\cY),\sU^{-1}(\cY)\bigr)
\]
in the \(\eta\)-extension. Since exact equivalences preserve
\(\Ext^1\)-orthogonals,
\[
\Psi\bigl({}^\perp\sU^{-1}(\cY)\bigr)
={}^\perp\cY_\Gamma.
\]
Applying Lemma \ref{lem:exact-equivalence} once more yields
$
\bigl({}^\perp\cY_\Gamma,\cY_\Gamma\bigr)
$
as a hereditary complete cotorsion pair in
\(\Gamma\text{-}\Mod\).
\end{proof}

\begin{corollary}\label{cor:comma}
Let \(\cC,\cD\) be abelian categories and let
\(G:\cC\to\cD\) be right exact. Let \((\cU,\cX)\) and
\((\cV,\cY)\) be right \(n\)-cotorsion pairs in \(\cC\) and
\(\cD\), respectively. Define
\[
\mathfrak P_G(\cU,\cV)=
\left\{
\begin{pmatrix}U\\D\end{pmatrix}_{\alpha}
\ \middle|\
U\in\cU,\ \alpha:G(U)\to D\text{ is monic},
\ \Coker\alpha\in\cV
\right\},
\]
and
\[
\mathfrak A_G(\cX,\cY)=
\left\{
\begin{pmatrix}X\\Y\end{pmatrix}_{\alpha}
\ \middle|\
X\in\cX,\ Y\in\cY
\right\}.
\]
Assume
$
L_iG(\cU)=0 (1\leq i\leq n+1),
$
that \(\mathfrak A_G(\cX,\cY)\) is coresolving, and that
$
\mathfrak P_G(\cU,\cV)^\vee_{n-1}
$
is closed under extensions. Then
$
\bigl(\mathfrak P_G(\cU,\cV),
\mathfrak A_G(\cX,\cY)\bigr)
$
is a right \(n\)-cotorsion pair in \((G\downarrow\cD)\).
\end{corollary}

\begin{proof}
Define
\[
\mathsf H:\cC\times\cD\to\cC\times\cD,\qquad
\mathsf H(C,D)=(0,G(C)).
\]
Then \(\mathsf H^2=0\), and there is an exact equivalence
$
(G\downarrow\cD)\simeq(\cC\times\cD)\ltimes\mathsf H,
$
see \cite[Remark 2.2(1)]{LiuFengHuZhang}. Hence this extension is
uniformly image-nilpotent of index at most two.

In Theorem \ref{thm:local-global-right-n}, take
$
\mathcal P=\mathfrak P_G(\cU,\cV), \mathcal R=\mathfrak A_G(\cX,\cY).
$
By \cite[Lemma 3.2(1)]{LiuFengHuZhang},
$
\mathcal R=\mathcal P^{\perp_{[1,n]}},
$
and hence
\[
\Ext^i_{(G\downarrow\cD)}(\mathcal P,\mathcal R)=0,
\qquad1\leq i\leq n.
\]
The equality also shows that \(\mathcal R\) is closed under direct
summands. Since \(\mathcal R\) is coresolving by hypothesis, it is
extension-closed, and the last assertion of Theorem
\ref{thm:local-global-right-n} gives
$
\Ext^2_{(G\downarrow\cD)}
(\mathcal P^\vee_{n-1},\mathcal R)=0.
$
Moreover, \(\mathcal P^\vee_{n-1}\) is extension-closed by hypothesis.

Under the above equivalence, the zero embedding of
\((C,D)\in\cC\times\cD\) is
$
\sZ(C,D)=\begin{pmatrix}C\\D\end{pmatrix}_0.
$
The construction in the proof of
\cite[Theorem 3.5(1)]{LiuFengHuZhang} gives, for every
\(C\in\cC\) and \(D\in\cD\), a short exact sequence
\[
0\to\begin{pmatrix}C\\D\end{pmatrix}_0
\to R_{C,D}\to P_{C,D}\to0
\]
with
$
R_{C,D}\in\mathcal R, P_{C,D}\in\mathcal P^\vee_{n-1}.
$
Thus all hypotheses of Theorem
\ref{thm:local-global-right-n} are satisfied. Therefore
\[
\bigl(\mathfrak P_G(\cU,\cV),
\mathfrak A_G(\cX,\cY)\bigr)
\]
is a right \(n\)-cotorsion pair in \((G\downarrow\cD)\). This
recovers \cite[Theorem 3.5(1)]{LiuFengHuZhang}.
\end{proof}

\begin{corollary}\label{cor:triangular}
Let
\[
\Lambda=\begin{pmatrix}A&0\\M&B\end{pmatrix}
\]
be a formal triangular matrix ring. Let \((\cU,\cX)\) and
\((\cV,\cY)\) be right \(n\)-cotorsion pairs in \(A\text{-}\Mod\)
and \(B\text{-}\Mod\), respectively. Set
\[
\mathfrak P_M(\cU,\cV)=
\left\{
\begin{pmatrix}U\\V'\end{pmatrix}_{\alpha}
\ \middle|\
U\in\cU,\ \alpha:M\otimes_AU\to V'\text{ is monic},
\ \Coker\alpha\in\cV
\right\},
\]
and
\[
\mathfrak A_M(\cX,\cY)=
\left\{
\begin{pmatrix}X\\Y\end{pmatrix}_{\alpha}
\ \middle|\
X\in\cX,\ Y\in\cY
\right\}.
\]
The following statements hold.
\begin{enumerate}
\item If $\Tor_i^A(M,\cU)=0 (1\leq i\leq n+1),$ \(\mathfrak A_M(\cX,\cY)\) is coresolving, and
\(\mathfrak P_M(\cU,\cV)^\vee_{n-1}\) is extension-closed, then
$
\bigl(\mathfrak P_M(\cU,\cV),
\mathfrak A_M(\cX,\cY)\bigr)
$
is a right \(n\)-cotorsion pair in \(\Lambda\text{-}\Mod\).

\item Suppose \(n=1\), \((\cU,\cX)\) and \((\cV,\cY)\) are hereditary complete
cotorsion pairs, and
$
\Tor_i^A(M,\cU)=0 (i\geq1).
$
Then
$
\bigl(\mathfrak P_M(\cU,\cV),
\mathfrak A_M(\cX,\cY)\bigr)
$
is a hereditary complete cotorsion pair in
\(\Lambda\text{-}\Mod\).
\end{enumerate}
\end{corollary}

\begin{proof}
For {\rm(1)}, take
$
\cC=A\text{-}\Mod, \cD=B\text{-}\Mod, G=M\otimes_A-.
$
There is an exact equivalence
\[
\Lambda\text{-}\Mod\simeq
(G\downarrow B\text{-}\Mod),
\]
and
\[
\mathfrak P_G(\cU,\cV)=\mathfrak P_M(\cU,\cV),
\qquad
\mathfrak A_G(\cX,\cY)=\mathfrak A_M(\cX,\cY).
\]
Moreover,
\[
L_iG(-)\cong\Tor_i^A(M,-).
\]
Thus the assumptions in {\rm(1)} are precisely those of Corollary
\ref{cor:comma}, and the first assertion follows. This recovers
\cite[Theorem 3.7]{LongZhang}.

For {\rm(2)}, take \(n=1\),
$
\mathcal P=\mathfrak P_M(\cU,\cV), \mathcal R=\mathfrak A_M(\cX,\cY).
$
Under the hypotheses, the results in \cite[Theorems 4.4(1) and 5.5]{MaoTriangular} verify the assumptions
of Theorem \ref{thm:local-global-right-n} and Corollary
\ref{cor:local-global-left-n}. Hence
\[
\bigl(\mathfrak P_M(\cU,\cV),
\mathfrak A_M(\cX,\cY)\bigr)
\]
is a complete cotorsion pair. Since
\(\mathfrak A_M(\cX,\cY)\) is coresolving, Proposition
\ref{prop:cotorsion-criteria}(2) shows that it is hereditary. Thus
\[
\bigl(\mathfrak P_M(\cU,\cV),
\mathfrak A_M(\cX,\cY)\bigr)
\]
is a hereditary complete cotorsion pair in
\(\Lambda\text{-}\Mod\), recovering
\cite[Theorem 5.6(1)]{MaoTriangular}.
\end{proof}

\begin{corollary} Let $\cA=\cB\ltimes\sF$ be a right trivial extension, and let \((\cX,\cY)\) be a hereditary
complete cotorsion pair in \(\cB\). If
$
L_1\sF(\cX)=0, \sF(\cX)\subseteq\cY,
$
then
$
\bigl({}^\perp\sU^{-1}(\cY),\sU^{-1}(\cY)\bigr)
$
is a hereditary complete cotorsion pair in \(\cA\).
\end{corollary}

\begin{proof}
Since \(\eta=0\), one has \(\eta^{[2]}=0\). Hence Theorem
\ref{thm:nilpotence-criterion} shows that \(\cA\) is uniformly
image-nilpotent of index at most two. The assertion follows from
Theorem \ref{thm:complete-cotorsion}. Thus
\cite[Theorem 3.13]{Hu} is the special case
$
\eta=0, d=2
$
of Theorem \ref{thm:complete-cotorsion}.
\end{proof}

\begin{corollary}\label{cor:Morita}
Let
\[
\Lambda=\begin{pmatrix}A&N\\M&B\end{pmatrix}
\]
be a Morita context ring with zero context homomorphisms. Put
\[
\mathcal E=A\text{-}\Mod\times B\text{-}\Mod,\qquad
\sF(X,Y)=(N\otimes_BY,M\otimes_AX).
\]
There is an exact equivalence
$
\Lambda\text{-}\Mod\simeq\mathcal E\ltimes\sF.
$
Let \((\cU,\cX)\) and \((\cV,\cY)\) be right \(n\)-cotorsion pairs in
\(A\text{-}\Mod\) and \(B\text{-}\Mod\), respectively. The following
statements hold.
\begin{enumerate}
\item[(1)] If
\[
(\cU\times\cV,\cX\times\cY)
\]
is an \(\sF\)-compatible right \(n\)-cotorsion pair in \(\mathcal E\),
then, under the above equivalence,
\[
\bigl(\sT(\cU\times\cV),
\sU^{-1}(\cX\times\cY)\bigr)
\]
is a right \(n\)-cotorsion pair in \(\Lambda\text{-}\Mod\).

\item[(2)] If \((\cU,\cX)\) and \((\cV,\cY)\) are hereditary complete
cotorsion pairs satisfying
\[
\Tor^B_1(N,\cV)=0,\qquad
\Tor^A_1(M,\cU)=0,
\]
and
\[
N\otimes_B\cV\subseteq\cX,\qquad
M\otimes_A\cU\subseteq\cY,
\]
then, under the above equivalence,
\[
\bigl({}^{\perp}\sU^{-1}(\cX\times\cY),
\sU^{-1}(\cX\times\cY)\bigr)
\]
is a hereditary complete cotorsion pair in
\(\Lambda\text{-}\Mod\).
\end{enumerate}
\end{corollary}

\begin{proof}
The displayed equivalence is the standard realization of the module
category of a Morita context ring with zero context homomorphisms, see
\cite{Green,CuiRongZhang,Hu}. The associated natural transformation is
\(\eta=0\), because the two context homomorphisms vanish. Hence
\(\eta^{[2]}=0\), and Theorem \ref{thm:nilpotence-criterion} shows that
\(\mathcal E\ltimes\sF\) is uniformly image-nilpotent of index at most
two.

Part {\rm(1)} follows directly from Theorem
\ref{thm:lift-n-cotorsion}.

For {\rm(2)}, the product pair
$
(\cU\times\cV,\cX\times\cY)
$
is a hereditary complete cotorsion pair in \(\mathcal E\). Moreover,
\[
L_1\sF(\cU\times\cV)
\cong
\bigl(\Tor^B_1(N,\cV),\Tor^A_1(M,\cU)\bigr)=0,
\]
and
\[
\sF(\cU\times\cV)
=
(N\otimes_B\cV,M\otimes_A\cU)
\subseteq\cX\times\cY.
\]
Thus Theorem \ref{thm:complete-cotorsion} applies and gives the
asserted hereditary complete cotorsion pair.
\end{proof}

\end{document}